\documentclass[12pt,reqno]{amsart}

\usepackage{amssymb}

\newtheorem{theorem}{Theorem}[section]
\newtheorem{proposition}[theorem]{Proposition}
\newtheorem{lemma}[theorem]{Lemma}

\numberwithin{equation}{section}

\begin{document}

\title[Poisson structure on character varieties]{Poisson structure on character varieties}

\author[I. Biswas]{Indranil Biswas}

\address{School of Mathematics, Tata Institute of Fundamental Research,
Homi Bhabha Road, Mumbai 400005}

\email{indranil@math.tifr.res.in}

\author[L. C. Jeffrey]{Lisa C. Jeffrey}

\address{Department of Mathematics,
University of Toronto, Toronto, Ontario, Canada}

\email{jeffrey@math.toronto.edu}



\date{}

\begin{abstract}
We show that the character variety for a $n$-punctured oriented surface has a natural Poisson structure.
\bigskip

\noindent
{\it R\'esum\'e.}\,\, Nous d\'emontrons que la vari\'et\'e des caract\`eres d'une surface compacte
orient\'ee perc\'ee en $n$ points est dot\`ee d'une structure de Poisson naturelle.
\end{abstract}

\maketitle

\section{Introduction}\label{sec0}

Let $X$ be a compact connected oriented surface. Given any real or complex reductive Lie group $G$, the
character variety ${\mathcal R}(X,\, G)$ parametrizes the equivalence classes of
completely reducible $G$--homomorphisms of the fundamental group of $X$. Alternatively,
${\mathcal R}(X,\, G)$ parametrizes the isomorphism classes of
completely reducible flat $G$--connections on $X$. It is known that ${\mathcal R}(X,\, G)$
has a natural symplectic structure; this symplectic structure was constructed by
Atiyah--Bott and Goldman in \cite{AB}, \cite{G} respectively.

Fix finitely many points $\{x_1,\,\cdots,\, x_m\}$ of $X$, and fix a conjugacy class $C_i$
in $G$ for each $x_i$. Let ${\mathcal R}(X_0,\, G)$ be the character variety for
$X_0\,=\, X\setminus \{x_1,\,\cdots,\, x_m\}$. Let ${\mathcal R}(X_0,\, G)_C \, \subset\,
{\mathcal R}(X_0,\, G)$ be the locus of all flat $G$--connections on $X_0$ for which
the local monodromy around each $x_i$ lies in the conjugacy class $C_i$. It is
known that this subset is equipped with a natural symplectic structure \cite{BG}, \cite{GHJW}.
When $G$ is a compact group, and $X$ is equipped with a complex structure, then ${\mathcal R}(X_0,\,
G)_C$ is the moduli space of semistable parabolic $G_{\mathbb C}$--bundles \cite{MS},
where $G_{\mathbb C}$ is the complexification of $G$.

We prove that ${\mathcal R}(X_0,\, G)$ has a natural Poisson structure (see Section 
\ref{se3.2}). The above submanifolds ${\mathcal R}(X_0,\, G)_C$ of it equipped with symplectic 
structure are the symplectic leaves for this Poisson structure.

This result has been known for  many years -- for example, see the proof 
given in 
M. Audin's article (\cite{Au}, Theorem 2.2.1). However Audin's proof
proceeds by using loop groups and central extensions of the 
Lie algebra of a loop group. Our proof is much simpler; one of the reasons for it
is that we are able to use the known result that ${\mathcal R}(X_0,\, G)_C$ are
symplectic manifolds. 

\section{Tangent and cotangent bundles of character varieties}

Let $X$ be a compact connected oriented $C^\infty$ surface. Fix a nonempty finite
subset
$$
D\, :=\, \{x_1,\,\cdots,\, x_m\}\, \subset\, X\, .
$$
Let $X_0\, :=\, X\setminus D$ be the complement. Fix a base point $x_0\, \in\, X_0$. For
notational convenience, the fundamental group $\pi_1(X_0,\, x_0)$ will be denoted by $\Gamma$.

Let $G$ be a connected reductive algebraic Lie group, which is defined over
$\mathbb R$ or $\mathbb C$. This implies that 
the Lie algebra ${\mathfrak g}\, :=\, \text{Lie}(G)$
admits a $G$--invariant nondegenerate symmetric bilinear form. Fix a
$G$--invariant nondegenerate symmetric bilinear form
\begin{equation}\label{e2}
B\, \in\, \text{Sym}^2({\mathfrak g}^*)\, .
\end{equation}

Consider the character variety
\begin{equation}\label{e3}
{\mathcal R}\, :=\, {\mathcal R}(X_0,\, G) \,:=\, \text{Hom}(\Gamma,\, G)^0/G\, ,
\end{equation}
where $\text{Hom}(\Gamma,\, G)^0\, \subset\, \text{Hom}(\Gamma,\, G)$ is the
locus of homomorphisms with completely reducible image.
We note that the points of ${\mathcal R}$ correspond to the equivalence classes of
homomorphisms $\rho\, :\, \Gamma\, \longrightarrow\, G$ such that the Zariski closure of
$\rho(\Gamma)$ is a reductive subgroup of $G$.

Take any homomorphism $\rho\, :\, \Gamma\, \longrightarrow\, G$. Let
$E^\rho_G\, \longrightarrow\, X_0$ be the corresponding principal $G$--bundle on $X_0$
equipped with a flat connection.
We briefly recall the construction of the flat bundle $E^\rho_G$. Let $q_0\,:\,
(\widetilde{X}_0,\, \widetilde{x}_0)\, \longrightarrow\, (X_0,\, x_0)$ be
the universal cover of $X_0$ for the base point $x_0$.
The total space of $E^\rho_G$ is the quotient of $\widetilde{X}_0\times G$, where
two points $(x_1,\, g_1)$ and $(x_2,\, g_2)$ of $\widetilde{X}_0\times G$ are identified
if there is an element $\gamma\, \in\, \Gamma$ such that $x_2\,=\, x_1\gamma$ and
$g_2\,=\,\rho(\gamma)^{-1}g_1$ (the fundamental group $\Gamma$ acts on
$\widetilde{X}_0$ as deck transformations). The projection of $E^\rho_G$ to $X_0$ is given by the
map $(x,\, g)\, \longmapsto\, q_0(x)$. The action of $G$ on $\widetilde{X}_0\times G$
given by the right--translation action of $G$ on itself
produces an action of $G$ on the quotient space $E^\rho_G$, making $E^\rho_G$ a principal
$G$--bundle over $X_0$. The trivial connection on the trivial principal $G$--bundle
$\widetilde{X}_0\times G\, \longrightarrow\, \widetilde{X}_0$ descends to a flat connection
on the principal $G$--bundle $E^\rho_G$. This flat connection on $E^\rho_G$ will be
denoted by $\widehat{\nabla}^\rho$.

The flat connection $\widehat{\nabla}^\rho$ induces a flat connection on every fiber
bundle associated to the principal $G$--bundle $E^\rho_G$. In particular, it induces
a flat connection on the adjoint vector bundle $\text{ad}(E^\rho_G)$ associated
to $E^\rho_G$ for the adjoint action of $G$ on the Lie algebra $\mathfrak g$. This
induced flat connection on $\text{ad}(E^\rho_G)$ will be denoted by $\nabla^\rho$.

Let
\begin{equation}\label{e4}
\underline{\text{ad}}(E^\rho_G)\, \longrightarrow\, X_0
\end{equation}
be the locally constant sheaf on $X_0$ given by the sheaf of covariant constant sections of the
vector bundle $\text{ad}(E^\rho_G)$ for the flat connection $\nabla^\rho$. It is known
that the tangent spaces of ${\mathcal R}$ defined in \eqref{e3} have the
following description: For any $\rho\, \in\, {\mathcal R}$,
\begin{equation}\label{e5}
T_\rho{\mathcal R}\,=\, H^1(X_0,\, \underline{\text{ad}}(E^\rho_G))\, ,
\end{equation}
where $\underline{\text{ad}}(E^\rho_G)$ is constructed in \eqref{e4} \cite{G}, \cite{AB}. Since
$X_0$ is oriented, this gives the following description of the cotangent space:
\begin{equation}\label{e5b}
T^*_\rho{\mathcal R}\,=\, H^1(X_0,\, \underline{\text{ad}}(E^\rho_G))^*
\,=\, H^1_c(X_0,\, \underline{\text{ad}}(E^\rho_G)^*)\, ,
\end{equation}
where $H^i_c$ is the compactly supported cohomology \cite{GHJW}, \cite{G}, \cite{BG}, and
$\underline{\text{ad}}(E^\rho_G)^*$ is the dual local system.
The pairing between $H^1(X_0,\, \underline{\text{ad}}(E^\rho_G))$ and
$H^1_c(X_0,\, \underline{\text{ad}}(E^\rho_G)^*)$ is constructed in the following way:
$$
H^1_c(X_0,\, \underline{\text{ad}}(E^\rho_G)^*)\otimes H^1(X_0,\,\underline{\text{ad}}(E^\rho_G))
\,\longrightarrow\, H^2_c(X_0,\,\underline{\text{ad}}(E^\rho_G)^*\otimes
\underline{\text{ad}}(E^\rho_G))
$$
\begin{equation}\label{cp}
\longrightarrow\, H^2_c(X_0,\, k)\,=\, H^2_c(X,\, k)\,=\, k\, ,
\end{equation}
where $k$ is either $\mathbb R$ or $\mathbb C$ depending on whether the Lie group $G$
is real or complex.

The bilinear form $B$ in \eqref{e2}, being $G$--invariant, produces a fiberwise
symmetric nondegenerate bilinear form
$\widetilde{\mathcal B}\, \in\, C^\infty(X_0,\, \text{Sym}^2(\text{ad}(E^\rho_G)^*))$.
This section $\widetilde{\mathcal B}$ is clearly covariant constant with respect to the flat
connection on $\text{Sym}^2(\text{ad}(E^\rho_G)^*)$ induced by the above flat connection
$\nabla^\rho$ on $\text{ad}(E^\rho_G)$ associated to $\rho$. In other words, we have
$$
\widetilde{\mathcal B}\, \in\, H^0(X_0,\, \text{Sym}^2(\underline{\text{ad}}(E^\rho_G)^*))\, ;
$$
note that $\text{Sym}^2(\underline{\text{ad}}(E^\rho_G)^*)$ coincides with the local
system on $X_0$ defined by the sheaf of covariant constant sections of
$\text{Sym}^2(\text{ad}(E^\rho_G)^*)$.
Consequently, $\widetilde{\mathcal B}$ produces an isomorphism of local systems
\begin{equation}\label{e6}
{\mathcal B}\, :\, \underline{\text{ad}}(E^\rho_G)\, \stackrel{\sim}{\longrightarrow}\,
\underline{\text{ad}}(E^\rho_G)^*\, .
\end{equation}
Combining \eqref{e5b} and \eqref{e6}, we have
\begin{equation}\label{e7}
T^*_\rho{\mathcal R}\,=\, H^1_c(X_0,\, \underline{\text{ad}}(E^\rho_G))\, .
\end{equation}
For any sheaf $F$ on $X_0$, there is a natural homomorphism $H^1_c(X_0,\, F)\,
\longrightarrow\, H^1(X_0,\, F)$ given by the inclusion homomorphism of the
corresponding sheaves. In particular, we have a homomorphism
\begin{equation}\label{hc}
\widetilde{\Phi}_\rho\, :\, H^1_c(X_0,\, \underline{\text{ad}}(E^\rho_G)) \, \longrightarrow\,
H^1(X_0,\, \underline{\text{ad}}(E^\rho_G))\, ,
\end{equation}
which, using \eqref{e5} and \eqref{e7}, gives a homomorphism
$T^*_\rho{\mathcal R}\, \longrightarrow\,T_\rho{\mathcal R}$.
This homomorphism $T^*_\rho{\mathcal R}\, \longrightarrow\,T_\rho{\mathcal R}$ defines an element
$\Phi_\rho\, \in\, T_\rho{\mathcal R}\otimes T_\rho{\mathcal R}$.

\begin{lemma}\label{lem1}
The above element $\Phi_\rho$ lies in the subspace $\wedge^2 T_\rho{\mathcal R}\,
\subset\, T_\rho{\mathcal R}\otimes T_\rho{\mathcal R}$.
\end{lemma}

\begin{proof}
Take $\alpha,\, \beta\, \in\, H^1_c(X_0,\, \underline{\text{ad}}(E^\rho_G))\,=\,
T^*_\rho{\mathcal R}$. For the pairing $\langle - ,\,-\rangle$ in \eqref{cp}, we have
\begin{equation}\label{el}
\langle\alpha,\, \widetilde{\Phi}_\rho(\beta)\rangle\,=\, - \langle\beta,\,
\widetilde{\Phi}_\rho(\alpha)\rangle\, ,
\end{equation}
where $\widetilde{\Phi}_\rho$ is the homomorphism in \eqref{hc}; the isomorphism
in \eqref{e6} has been used in \eqref{el}. The lemma follows from \eqref{el}.
\end{proof}

The above pointwise construction of $\Phi_\rho$, being canonical, produces a section
\begin{equation}\label{e8}
\Phi\, \in\, C^\infty({\mathcal R},\, \wedge^2T{\mathcal R})\, .
\end{equation}

We will show, in the next section, that this $\Phi$ defines a Poisson structure on
$\mathcal R$.

\section{Poisson structure on $\mathcal R$}

\subsection{A criterion for Poisson structure}

Let $M$ be a smooth manifold. Let
$$
\Theta\, \in\, C^\infty(M,\, \wedge^2TM)
$$
be a smooth section. For any point $x\, \in\, M$, let
\begin{equation}\label{ty}
\Theta_x\, :\, T^*_xM\, \longrightarrow\, T_xM
\end{equation}
be the homomorphism defined by the equation $w(\Theta_x (v))\,=\, \Theta(x)(v\wedge w)$
for all $v,\, w\, \in\, T^*_xM$. The image
$$
V_x\, :=\, \Theta_x(T^*_xM)\, \subset\, T_xM
$$
is equipped with a symplectic form. To prove this, let
$$
\varphi\, :\, T^*_xM\, \longrightarrow\, V^*_x
$$
be the dual of the inclusion map $V_x\, \hookrightarrow\, T_xM$, so $\varphi$ is
surjective. We will prove that $\varphi$ vanishes on the subspace
$\text{kernel}(\Theta_x)\, \subset\, T^*_xM$. For any $v,\, w\, \in\, T^*_xM$, we have
\begin{equation}\label{sk}
w(\Theta_x(v))+ v(\Theta_x(w))\,=\, 0\, ,
\end{equation}
because $\Theta$ is skew-symmetric. So if $v\, \in\, \text{kernel}(\Theta_x)$, then
$v(\Theta_x(w))\,=\, 0$ for all $w$, implying that $v(V_x)\,=\, 0$. So
$\varphi$ vanishes on $\text{kernel}(\Theta_x)$, and hence it
descends to a homomorphism
$$
\widehat{\varphi}_x\, :\, \frac{T^*_xM}{\text{kernel}(\Theta_x)}\,=\, \text{image}(\Theta_x)
\,=\, V_x\, \longrightarrow\, V^*_x\, .
$$
{}From \eqref{sk} it follows that $\widehat{\varphi}'_x\, \in\, \wedge^2 V^*_x$, where
$\widehat{\varphi}'_x$ is the bilinear form on $V_x$ defined by $\widehat{\varphi}_x$. Since
$\varphi$ is surjective, it follows that this $\widehat{\varphi}_x$ is
also surjective. This implies that the bilinear form $\widehat{\varphi}'_x$ is nondegenerate.
So, $\widehat{\varphi}'_x$ is a symplectic form on $V_x$.

The section $\Theta$ is a called a Poisson structure if the Schouten--Nijenhuis bracket
$[\Theta,\, \Theta]$ vanishes identically \cite{Ar}. An equivalent formulation of
this definition is the following: Given a pair of $C^\infty$ functions $f_1$ and $f_2$
on $M$, define the function
$$
\{f_1,\, f_2\}\,=\, \Theta((df_1)\wedge (df_2))\, .
$$
Then $\Theta$ is Poisson if and only if
\begin{equation}\label{ji}
\{f_1,\, \{f_2,\, f_3\}\}+ \{f_2,\, \{f_3,\, f_1\}\} + \{f_3,\, \{f_1,\, f_2\}\}\,=\, 0
\end{equation}
for all $C^\infty$ functions $f_1,\, f_2,\, f_3$.

Let $\widehat{\Theta}\, :=\, T^*M\, \longrightarrow\, TM$ be the homomorphism
given by $\Theta$, so $\widehat{\Theta}(x)\,=\, \Theta_x$ for every $x\, \in\, X$.

The following proposition is in the literature already. For example, see Proposition 1.8 and 
Remark 1.10 of the notes \cite{mein} from E. Meinrenken's 2017 graduate course on Poisson 
geometry. However we have included a proof for completeness.

\begin{proposition}\label{prop1}
The section $\Theta$ is Poisson if and only if the following two hold:
\begin{enumerate}
\item The subsheaf $\widehat{\Theta}(T^*M)\, \subset\, TM$ is closed under the
Lie bracket operation.

\item For any leaf $\mathbb L$ of the nonsingular locus of the integrable
distribution $\widehat{\Theta}(T^*M)$, let $\widehat{\varphi}'_{\mathbb L}$
be the two--form on $\mathbb L$ defined by the equation
$$
\widehat{\varphi}'_{\mathbb L}(x)(v_1,\, v_2)\,=\, \widehat{\varphi}'_x(v_1,\, v_2)
$$
for all $x\, \in\, \mathbb L$, and $v_1,\, v_2\, \in\, T_x\mathbb L$, where $\widehat{\varphi}'_x$
is constructed above.
Then $\widehat{\varphi}'_{\mathbb L}$ is a symplectic form on $\mathbb L$.
\end{enumerate}
\end{proposition}

\begin{proof}
If $\Theta$ is Poisson, it is standard that the above two conditions hold. We shall
prove the converse.

Take any point $x\, \in\, M$ where the distribution $\widehat{\Theta}(T^*M)$ is nonsingular,
meaning the dimension of the subspace $\widehat{\Theta}(y)(T^*_yM)\, \subset\, T_yM$ is
unchanged for all points $y$ in an open neighborhood of $x$. Let ${\mathbb L}$ denote
the leaf, passing through $x$, of the foliation restricted to a sufficiently small open neighborhood
$U$ of $x$ in $M$. For any two smooth functions $f,\, g$ defined on $U$, consider the function
$$
\{f,\, g\}\, :\, {\mathbb L}\, \longmapsto\, {\mathbb R}\, ,\ \
y\, \longmapsto\, \Theta(y)(df(y),\, dg(y))\, \in\, \mathbb R\, .
$$
Let $h$ be a smooth function defined on $U$ such that $h\vert_{\mathbb L}\,=\,
f\vert_{\mathbb L}$. We will prove that
\begin{equation}\label{s1}
\{f,\, g\}\,=\, \{h,\, g\}\, .
\end{equation}

Note that $f-h$ vanishes on $\mathbb L$. For notational convenience, denote the function $f-g$ 
by $\delta$. To prove \eqref{s1}, consider the function $\{f-h,\, g\}$. For any $y\, \in\, 
{\mathbb L}$, we have
\begin{equation}\label{e9}
\{\delta,\, g\}(y) \,=\, dg(y)(\Theta_y(d\delta(y)))\,=\, -d\delta(y) (\Theta_y(dg(y)))\, ,
\end{equation}
where $\Theta_y$ is constructed as in \eqref{ty}. The pullback of the $1$--form $d\delta$ to the
submanifold ${\mathbb L}\, \subset\, U$ vanishes identically because the restriction of $\delta$
to $\mathbb L$ is identically zero. On the other hand, the tangent vector $\Theta_y(dg(y))\,\in\,
T_y U$ lies in the subspace ${\mathbb L}_y\, \subset\, T_yU$ (recall that ${\mathbb L}_y\,=\,
\Theta_y(T^*_y M)$). Therefore, we have
$$
d\delta(y) (\Theta_y(dg(y)))\, =\, 0\, .
$$
Hence $\{\delta,\, g\}(y)\,=\, 0$ by \eqref{e9}. This proves \eqref{s1}.

In view of \eqref{s1} to prove that the Poisson bracket $\{-,\, -\}$ satisfies the
Jacobi identity in \eqref{ji}, it suffices to show that the Jacobi identity is satisfied
by the Poisson bracket operation on functions on a leaf $\mathbb L$, where the Poisson bracket is
defined using the nondegenerate two--form on the leaf
given by $\widehat{\varphi}'_{\mathbb L}$. But condition (2) in the
proposition says that $\widehat{\varphi}'_{\mathbb L}$ is symplectic on a leaf, and hence the
Poisson bracket on a leaf satisfies the Jacobi identity. This completes the proof of the proposition.
\end{proof}

\subsection{Application of the criterion}\label{se3.2}

Using Proposition \ref{prop1}, it will be shown that $\Phi$ constructed in \eqref{e8}
is a Poisson structure on $\mathcal R$.

Consider the homomorphism $\Phi_1\, :\, T^*{\mathcal R}\, \longrightarrow\, T{\mathcal R}$
constructed from $\Phi$ in \eqref{e8} as follows:
$$
v(\Phi_1(\rho)(w)) \,=\, \Phi(\rho)(w\wedge v)
$$
for all $v,\, w\, \in\, T^*_\rho{\mathcal R}$ and $\rho\, \in\, \mathcal R$.
So, $\Phi_1(\rho)\, :\, T^*_\rho{\mathcal R}\, \longrightarrow\, T_\rho{\mathcal R}$,
$\rho\, \in\, \mathcal R$, coincides with
the homomorphism $\widetilde{\Phi}_\rho$ in \eqref{hc}. Therefore, the image of
$\Phi_1$ corresponds to the foliation on $\mathcal R$ given by loci with fixed conjugacy
classes for the punctures $\{x_1,\, \cdots,\, x_m\}$. In particular, the distribution
$\Phi_1(T^*{\mathcal R})$ is integrable; so the first condition
in Proposition \ref{prop1} is satisfied. On each leaf the two--form is symplectic \cite{GHJW}, \cite{BG}, \cite{G};
so the second condition in Proposition \ref{prop1} is also satisfied.

\section{Extended moduli space}

It is shown in \cite{FM} (Theorem 4.3) that the quotient of a symplectic manifold by a group 
action preserving the symplectic structure is a Poisson manifold. We may apply this to the 
symplectic manifold given in Section 2.3 of \cite{J} (the extended moduli space, which is a 
symplectic quotient of the space of all connections on a vector bundle over an oriented 
2-manifold by the based gauge group). The symplectic structure on the extended moduli space is 
given in Section 3.1 of \cite{J}.

In this section only, let $G$ be a compact connected Lie group. The extended moduli space 
$M_{ext}$ may be written as the push-out of the Lie algebra of $G$ and the space of 
representations $M \,=\, {\rm Hom} (\Gamma,\, G)$ of the fundamental group $\Gamma$ of a surface 
with one boundary component, where the map from the Lie algebra to $G$ is the exponential map, 
and the map from the space of flat connections to $G$ is the holonomy around the boundary 
component. In the case of one boundary component, this is summarized by the following 
commutative diagram. Let $M$ be the space ${\rm Hom}(\Gamma, G) $. The symplectic structure is 
defined in \cite{J} in terms of gauge equivalence classes of flat connections. The map ${\rm 
Hol}$ denotes the holonomy of the connection around the boundary component. The symplectic 
structure on ${\rm Hom} (\Gamma, G)/G$ is described from the point of view of representations of 
the fundamental group $\Gamma$ in the work of Goldman \cite{G}, \cite{G2}.
$$
\begin{matrix}
M_{ext} & \longrightarrow & {\mathfrak g}\\
\Big\downarrow& & {\exp}\Big\downarrow~\,~\,~\,~\,~\, \\
M & \stackrel{{\rm Hol}}{\longrightarrow} & G
\end{matrix}
$$
At a regular point, the extended moduli space is a cover of the representation
space ${\rm Hom}(\Gamma, G)$ with fiber the integer lattice of $G$
(the kernel of the exponential map).

For the case of multiple boundary components, we refer to Hurtubise-Jeffrey \cite{HJ}, 
Hurtubise-Jeffrey-Sjamaar \cite{HJS} and Huebschmann \cite{H}.

The description in Section 3.1 of \cite{J} establishes that $M_{ext}$ is symplectic (where it is 
smooth). At points in $\mathfrak g$ where the exponential map is a diffeomorphism, there is a 
local diffeomorphism between an open neighbourhood in $M_{ext}$ and an open neighbourhood of 
${\rm Hom}(\Gamma, G)$. Taking the quotient of $M_{ext}$ by $G$ we thus conclude that ${\rm 
Hom}(\Gamma, G)/G $ is a Poisson manifold.



\begin{thebibliography}{AAAA}

\bibitem{Ar} V. I. Arnol'd, {\it Mathematical methods of classical mechanics}.
Translated from the Russian by K. Vogtmann and A.
Weinstein, Second edition. Graduate Texts in Mathematics, 60. Springer-Verlag, New York, 1989.

\bibitem{At} M. F. Atiyah, Complex analytic connections in fibre
bundles, \textit{Trans. Amer. Math. Soc.} \textbf{85} (1957), 181--207.

\bibitem{AB} M. F. Atiyah and R. Bott, The Yang-Mills equations over Riemann surfaces,
{\it Philos. Trans. Roy. Soc. London} {\bf 308} (1983), 523--615.

\bibitem{Au} M. Audin, Lectures on gauge theory and integrable systems.
In {\em Gauge Theory and Symplectic Geometry} (J. Hurtubise, F. Lalonde,
ed., Kluwer, 1997), 1--48.

\bibitem{BG} I. Biswas and K. Guruprasad, Principal bundles on open surfaces and
invariant functions on Lie groups, {\it Internat. Jour. Math.} {\bf 4} (1993), 535--544.

\bibitem{GHJW} K. Guruprasad, J. Huebschmann, L. Jeffrey and A. Weinstein,
Group systems, groupoids, and moduli spaces of
parabolic bundles, {\it Duke Math. Jour.} {\bf 89} (1997), 377--412.

\bibitem{FM} R. L. Fernandes and I. Marcut, {\em Lectures
on Poisson geometry}, Springer, 2015.

\bibitem{G} W. Goldman, The symplectic nature of fundamental groups of
surfaces, {\em Adv. Math.} {\bf 54} (1984), 200--225.

\bibitem{G2} W. Goldman, Invariant functions on Lie groups and
Hamiltonian flows of surface group representations, {\em Invent. Math.}
{\bf 85} (1986), 263 --302.

\bibitem{H} J. Huebschmann, On the variation of the Poisson structures of certain moduli spaces, 
{\em Math. Ann.} {\bf 319} (2001), 267--310.

\bibitem{HJ} J. Hurtubise and L. Jeffrey, Representations with weighted frames and framed parabolic
bundles, {\em Canadian Jour. Math.} {\bf 52} (2000), 1235--1268.

\bibitem{HJS} J. Hurtubise, L. Jeffrey and R. Sjamaar, Moduli of Framed Parabolic
Sheaves, {\em Ann. Global Anal. Geom.} {\bf 28} (2005), 351--370.

\bibitem{J} L. Jeffrey, Extended moduli spaces of flat connections on 
Riemann surfaces. {\it Math. Annalen} {\bf 298} (1) (1994), 667--692.

\bibitem{MS} V. B. Mehta and C. S. Seshadri, Moduli of vector bundles on curves
with parabolic structure, {\it Math. Ann.} {\bf 248} (1980), 205--239.

\bibitem{mein} E. Meinrenken, {\em Introduction to Poisson Geometry},
\begin{verbatim}
http://www.math.toronto.edu/mein/teaching/MAT1341_PoissonGeometry/Poisson8.pdf
\end{verbatim}

\end{thebibliography}
\end{document}